\documentclass[11pt,a4paper]{article}

\usepackage{amsmath}
\usepackage{amsfonts}
\usepackage{amssymb}
\usepackage{centernot}
\usepackage[margin=2.7cm]{geometry}
\usepackage[colorlinks=true,linkcolor=blue,citecolor=blue,urlcolor=blue]{hyperref}
\usepackage{enumitem}
\usepackage{graphicx,fancybox}
\usepackage{tikz}
\usepackage{subfigure}
\usepackage[capitalise]{cleveref}
\usepackage{pgf}
\usepackage[font={small}]{caption}

\DeclareMathOperator*{\argmin}{argmin}
\DeclareMathAlphabet\mbc{OMS}{cmsy}{b}{n}

\newcommand{\sri}{\operatorname{sri}}
\newcommand{\gra}{\operatorname{gra}}
\newcommand{\zer}{\operatorname{zer}}
\newcommand{\Fix}{\operatorname{Fix}}
\newcommand{\ran}{\operatorname{ran}}
\newcommand{\dom}{\operatorname{dom}}
\newcommand{\Id}{\operatorname{Id}}
\newcommand{\Prox}{\operatorname{prox}}

\newcommand{\bs}{\boldsymbol}

\newcommand{\Hi}{\mathcal{H}}
\newcommand{\bHi}{\mbc{H}}
\newcommand{\R}{\mathbb{R}}

\def\tto{\rightrightarrows}
\Crefname{fact}{Fact}{Facts}
\Crefname{enumi}{}{}

\usepackage{amsthm}
\newtheorem{theorem}{Theorem}[section]
\newtheorem{definition}{Definition}[section]
\newtheorem{proposition}{Proposition}[section]
\newtheorem{corollary}{Corollary}[section]

\newtheorem{lemma}{Lemma}[section]
\newtheorem{fact}{Fact}[section]
\newtheorem{remark}{Remark}[section]
\newtheorem{example}{Example}[section]

\title{Computing the resolvent of the sum\\ of maximally monotone operators with the \\ averaged alternating modified reflections algorithm}

\author{Francisco J. Arag\'on Artacho\thanks{Department of Mathematics,
University of Alicante, \textsc{Spain}. e-mail:~\href{mailto:francisco.aragon@ua.es}{francisco.aragon@ua.es}}
        \and Rub\'en Campoy\thanks{Department of Mathematics,
University of Alicante, \textsc{Spain}. e-mail:~\href{mailto:ruben.campoy@ua.es}{ruben.campoy@ua.es}}
}

\begin{document}
\maketitle

\begin{abstract}
The averaged alternating modified reflections algorithm is a projection method for finding the closest point in the intersection of closed convex sets to a given point in a Hilbert space. In this work, we generalize the scheme so that it can be used to compute the resolvent of the sum of two maximally monotone operators. This gives rise to a new splitting method, which is proved to be strongly convergent. A standard product space reformulation permits to apply the method for computing the resolvent of a finite sum of maximally monotone operators. Based on this, we propose two variants of such parallel splitting method.
\end{abstract}



\section{Introduction}

The \emph{averaged alternating modified reflections (AAMR) algorithm} is a projection method that was recently  introduced in~\cite{AAMR} for solving best approximation problems in the convex setting. For the case of two nonempty, closed and convex sets $C_1$ and $C_2$ in a Hilbert space $\Hi$ with $C_1\cap C_2\neq\emptyset$, the corresponding best approximation problem consists in finding the closest point to a given point~$q\in\Hi$ in their intersection $C_1\cap C_2$, i.e.,
\begin{equation}\label{eq:bestaproxprob}
\text{Find } p\in C_1\cap C_2\text{ such that } \|p-q\|=\inf_{x\in C_1\cap C_2} \|x-q\|.
\end{equation}
For any initial point $x_0\in\Hi$, the AAMR algorithm is iteratively defined by
\begin{equation}\label{eq:aamr_intro}
x_{n+1}:=(1-\alpha)x_n+\alpha(2\beta P_{C_2-q}-\Id)(2\beta P_{C_1-q} - \Id)(x_n), \quad n=0,1,2,\ldots,
\end{equation}
where $P_C$ and $\Id$ denote the projector onto the set $C$ (see~\cref{ex:prox_proj}\cref{ex:proj}) and the identity mapping, respectively. When $\alpha,\beta\in{]0,1[}$, under the constraint qualification
\begin{equation}\label{eq:strongCHIP_intro}
q\in (\Id+N_{C_1}+N_{C_2})(P_{C_1\cap C_2}(q)),
\end{equation}
where $N_{C_1}$ and $N_{C_2}$ denote the normal cones (see~\cref{ex:subdif_normalcone}\cref{ex:normalcone}) to $C_1$ and $C_2$, respectively, the generated sequence $(x_n)_{n=0}^\infty$ is weakly convergent to a point $x^\star$ such that $P_{C_1}(x^\star+q)=P_{C_1\cap C_2}(q)$, which solves problem~\eqref{eq:bestaproxprob}. Furthermore, the \emph{shadow sequence} $\left(P_{C_1}(x_n+q)\right)_{n=0}^\infty$ is strongly convergent to the solution $P_{C_1\cap C_2}(q)$ of~\eqref{eq:bestaproxprob}, see~\cite[Theorem~4.1]{AAMR}.

The rate of convergence of the AAMR algorithm for the case of two subspaces has been recently analyzed in~\cite{AAMR_rate}. If the algorithm is run with an optimal selection of its parameters $\alpha$ and $\beta$, its rate of convergence was shown to be better than the one of other projection methods. In a more practical context, the AAMR algorithm has been recently employed in~\cite{BBK18} to solve a continuous-time optimal control problem, under the name Arag\'on Artacho--Campoy algorithm (AAC). Their numerical results show a very good performance of the algorithm, compared to the other methods considered.

The AAMR algorithm can be viewed as a modification of the so-called \emph{Douglas--Rachford (DR)} algorithm~\cite{DR56} (also known as \emph{averaged alternating reflections method}), which is defined as in~\eqref{eq:aamr_intro} for $\beta=1$ and $q=0$. This iterative method only solves feasibility problems of the form
\begin{equation}\label{eq:feasprob}
\text{Find } x\in C_1\cap C_2,
\end{equation}
rather than best approximation problems of the type~\eqref{eq:bestaproxprob}. With no constraint qualification needed, the sequence generated by DR is weakly convergent to a point $x^\star$ such that $P_{C_1}(x^\star)\in C_1\cap C_2$, which thus solves~\eqref{eq:feasprob}. In this case, the shadow sequence $\left(P_{C_1}(x_n)\right)_{n=0}^\infty$ is only proved to be weakly convergent to the solution $P_{C_1}(x^\star)$, see~\cite{Svaiter}.

The Douglas--Rachford scheme can be more generally applied to monotone operators~\cite{LM79}. In this context, the DR algorithm can be used to solve problems of the form
\begin{equation}\label{eq:zersumprob}
\text{Find } x\in\zer(A+B)=\{x\in\Hi \mid 0\in Ax+Bx\},
\end{equation}
where $A,B:\Hi\tto\Hi$ are maximally monotone operators. The general structure of the iteration is the same as in the feasibility context, but replacing the projectors onto the sets with the resolvents $J_A$ and $J_B$ of the operators (see~\cref{def:resolvent}), i.e.,
\begin{equation}\label{eq:DRsplit_intro}
x_{n+1}:=(1-\alpha)x_n+\alpha(2J_{A}-\Id)(2J_{B} - \Id)(x_n), \quad n=0,1,2,\ldots.
\end{equation}
In fact, the feasibility problem~\eqref{eq:feasprob} can be written in the form~\eqref{eq:zersumprob} by taking $A=N_{C_1}$ and $B=N_{C_2}$. Since the resolvent of a normal cone to a convex set coincides with the projector onto the set, then~\eqref{eq:DRsplit_intro} becomes the DR iteration for solving feasibility problems.

The objective of this work is to extend the AAMR scheme to the more general context of maximally monotone operators. Given a point $q$ in the domain of $J_{A+B}$ (i.e., in the range of $A+B+\Id$), the generalized version of the best approximation problem~\eqref{eq:bestaproxprob} can be stated as,
\begin{equation}\label{eq:sumprob}
\text{Find } p=J_{A+B}(q),
\end{equation}
for some maximally monotone operators $A,B:\Hi\tto\Hi$. This is indeed a generalization of the best approximation problem~\eqref{eq:bestaproxprob}. Note that, if the constraint qualification~\eqref{eq:strongCHIP_intro} holds, we have that
\begin{equation*}
P_{C_1\cap C_2}(q)=J_{N_{C_1}+N_{C_2}}(q),
\end{equation*}
and thus~\eqref{eq:sumprob} becomes~\eqref{eq:bestaproxprob}.

The AAMR method can be naturally extended from the convex feasibility framework to the context of maximally monotone operators by considering modified reflectors instead of reflectors in the Douglas--Rachford splitting algorithm~\eqref{eq:DRsplit_intro}, i.e.,
\begin{equation}\label{eq:AAMRsplit_intro}
x_{n+1}:=(1-\alpha)x_n+\alpha(2\beta J_{B}-\Id)(2\beta J_{A} - \Id)(x_n), \quad n=0,1,2,\ldots,
\end{equation}
with $\beta\in{]0,1[}$. The analysis of AAMR for monotone operators~\eqref{eq:AAMRsplit_intro} presented in this work is inspired by the work of Combettes~\cite{C09}, where a different iterative construction of the resolvent of the sum is presented. Our analysis consists in reformulating the AAMR iteration so that it can be viewed as the one generated by the DR splitting algorithm for finding a zero of the sum of an appropriate modification of the operators. Another iterative approach can be found in~\cite{BC08}, where a Dykstra-like algorithm is developed. In~\cite{C11}, or the more recent work~\cite{ABC17}, the particular case of proximity mappings (see~\cref{ex:prox_proj}\cref{ex:prox}) is tackled.

The remainder of the paper is structured as follows. We give a short overview in \cref{sec:prelim} of some preliminary concepts and basic results about monotone operators. 
The extension of the AAMR method for computing the resolvent of the sum of two maximally monotone operators is given in~\cref{sec:AAMR}. Finally, in~\cref{sec:Parallel}, we use a product space reformulation to derive two different parallel splitting versions of the method to deal with an arbitrary finite family of operators.

\section{Preliminaries}\label{sec:prelim}

Throughout this paper, $\Hi$ is a real Hilbert space equipped with inner product $\langle\cdot , \cdot\rangle$ and induced norm $\|\cdot\|$. We abbreviate \emph{norm convergence} of sequences in $\Hi$ with $\to$  and we use $\rightharpoonup$ for \emph{weak convergence}. Given a set-valued operator $A:\Hi\tto\Hi$, the \emph{graph}, the \emph{domain}, the \emph{range}, the set of \emph{fixed points} and the set of \emph{zeros} of A, are denoted, respectively, by $\gra A$, $\dom A$, $\ran A$, $\Fix A$ and $\zer A$; i.e.,
\begin{gather*}
\gra A:=\left\{(x,u)\in\Hi\times\Hi : u\in A(x)\right\},\quad \dom A:=\left\{x\in\Hi : A(x)\neq\emptyset\right\},\\
\ran A:=\left\{x\in\Hi : x\in A(z) \text{ for some } z\in\Hi  \right\},\\
\Fix A:=\left\{x\in\Hi : x\in A(x)\right\}\quad
\text{and} \quad \zer A:=\left\{x\in\Hi : 0\in A(x)\right\}.
\end{gather*}

\begin{definition}
An operator $A:\Hi\tto\Hi$ is said to be
\begin{enumerate}[label=(\roman*)]

\item \emph{monotone} if
\begin{equation*}
\langle x-y,u-v\rangle\geq 0,\quad\forall (x,u),(y,v)\in\gra A;
\end{equation*}
\item \emph{maximally monotone} if it is monotone and there exists no monotone operator $B:\Hi\tto\Hi$ such that $\gra B$ properly contains $\gra A$; i.e., for every $(x,u)\in\Hi\times\Hi$,
\begin{equation*}
(x,u)\in\gra A \quad \Leftrightarrow \quad \langle x-y,u-v\rangle\geq 0,\; \forall (y,v)\in\gra A;
\end{equation*}
\item $\mu$-\emph{strongly monotone} for $\mu>0$, if $A-\mu I$ is monotone; i.e.,
\begin{equation*}
 \langle x-y,u-v\rangle\geq \mu\|x-y\|^2,\quad \forall (x,u),(y,v)\in\gra A.
\end{equation*}
\end{enumerate}
\end{definition}
Two well-known examples of maximally monotone operators are given next.
\begin{example}[The subdifferential and the normal cone operators] \label{ex:subdif_normalcone}\hfill
\begin{enumerate}[label=(\roman*)]
\item Let $f:\Hi\to{]-\infty,+\infty]}$ be a proper, lower semicontiuous and convex function. The \emph{subdifferential} of $f$, which is the operator $\partial f:\Hi\tto\Hi$ defined by
\begin{equation*}
\partial f(x):=\left\{ u\in\Hi :  \langle y-x, u\rangle +f(x)\leq f(y), \quad \forall y\in\Hi\right\},
\end{equation*}
is maximally monotone (see, e.g.,~\cite[Theorem~20.40]{BC11}).\label{ex:subdif}

\item Let $C$ be a nonempty, closed and convex subset of $\Hi$. The \emph{normal cone} to $C$, which is the operator $N_C:\Hi\tto\Hi$ defined by
\begin{equation*}
N_C(x):=\left\{\begin{array}{ll}\{u\in\Hi : \langle u, c-x \rangle\leq 0, \, \forall c\in C \}, &\text{if }x\in C,\\
\emptyset, & \text{otherwise,}\end{array}\right.
\end{equation*}
is maximally monotone (see, e.g.,~\cite[Example~20.41]{BC11}).\label{ex:normalcone}
\end{enumerate}
\end{example}

The following lemma shows the preservation  of (maximal) monotonicity under affine transformations. The proof is straightforward and omitted for brevity.

\begin{lemma}\label{fact:prop_monotone}
Let $A:\Hi\tto\Hi$ be (maximally) monotone, let ${w,z\in\Hi}$ and let $\gamma,\lambda\in\R$ such that $\gamma\lambda>0$. Then, the operator $\widetilde{A}:\Hi\tto\Hi$, defined for any $x\in\Hi$ by
$$\widetilde{A}(x):=w+\gamma A(\lambda x+z),$$
is (maximally) monotone.
\end{lemma}

A very useful characterization of maximal monotonicity is provided by the following fundamental result due to Minty~\cite{Minty}.

\begin{fact}[Minty's theorem]\label{fact:Minty}
Let $A:\Hi\tto \Hi$ be monotone. Then,
$$A \text{ is maximally monotone} \quad \Leftrightarrow \quad \ran(\Id+A)=\Hi.$$
\end{fact}
\begin{proof}
See, e.g., \cite[Theorem~21.1]{BC11}.
\end{proof}

Next we recall the definition of the resolvent of an operator, which is an important tool in the theory of monotone operators.

\begin{definition}\label{def:resolvent}
Let $A:\Hi\tto\Hi$ be an operator. The \emph{resolvent} of $A$ is $J_A:=(\Id+A)^{-1}$; i.e.,
\begin{equation*}
J_A(x)=\left\{y\in\Hi : x\in y+A(y)\right\},\quad\text{for all } x\in\Hi.
\end{equation*}
The \emph{reflected resolvent} is defined by $R_A:=2J_A-\Id$.
\end{definition}

Clearly, $\dom J_{A}=\ran(\Id+A)$, and thus Minty's theorem (\Cref{fact:Minty}) guarantees that the resolvent has full domain precisely when $A$ is  maximally monotone. In the following result we collect some additional properties regarding the single-valuedness  and nonexpasiveness of the resolvent and the reflected resolvent of maximally monotone operators.

\begin{fact}\label{fact:nonexpansive}
Let $A:\Hi\tto\Hi$ be a maximally monotone operator. Then,
\begin{enumerate}[label=(\roman*)]
\item $J_A:\Hi\mapsto\Hi$ is firmly nonexpansive, i.e.,
	\begin{equation*}
	\|J_A(x)-J_A(y)\|^2+\|(\Id-J_A)(x)-(\Id-J_A)(y)\|^2\leq\|x-y\|^2, \quad \forall x,y\in \Hi;
	\end{equation*}
\item $R_A:\Hi\mapsto\Hi$ is nonexpansive, i.e.,
	\begin{equation*}
	\|R_A(x)-R_A(y)\|\leq\|x-y\|, \quad \forall x,y\in \Hi.
	\end{equation*}
\end{enumerate}
\end{fact}
\begin{proof}
See, e.g.,~\cite[Corollary 23.10]{BC11}.
\end{proof}
The resolvents of the maximally monotone operators considered in~\Cref{ex:subdif_normalcone} are also some well-known mappings, as we show next.
\begin{example}[The proximity and the projector operators]\label{ex:prox_proj} The resolvents of the operators considered in~\Cref{ex:subdif_normalcone} are single-valued and firmly nonexpansive with full domain, according to~\Cref{fact:nonexpansive}.
\begin{enumerate}[label=(\roman*)]
\item Let $\partial f:\Hi\tto\Hi$ be the subdifferential of a proper, lower semicontiuous  and convex function $f:\Hi\to{]-\infty,+\infty]}$. Then, $J_{\partial f}=\Prox_f$, where $\Prox_f:\Hi\to\Hi$ is the \emph{proximity operator} of $f$ defined by
\begin{equation*}
\Prox_f(x):=\argmin_{u\in\Hi} \left( f(u)+\frac{1}{2}\|x-u\|^2\right), \quad \text{for all } x\in\Hi;
\end{equation*}
see, e.g.,~\cite[Example~23.3]{BC11}.\label{ex:prox}
\item Let $N_C$ be the normal cone to a nonempty, closed and convex set $C\subseteq\Hi$. Then, $J_{N_C}=P_C$, where $P_C:\Hi\to\Hi$ denotes the \emph{projector} onto $C$ defined by
\begin{equation*}
P_C(x):=\argmin_{c\in C}  \|x-c\|, \quad \text{for all } x\in\Hi;
\end{equation*}
see, e.g.,~\cite[Example~23.4]{BC11}.\label{ex:proj}
\end{enumerate}
\end{example}

We recall next the concept of perturbation of an operator, which was originally introduced and discussed in~\cite{BHM14}. We follow the notation used in~\cite{BM17}.

\begin{definition}
Let $A:\Hi\tto\Hi$ and let $w\in\Hi$. The corresponding \emph{inner $w$-perturbation} of $A$ is the operator $A_w:\Hi\tto\Hi$ defined by
\begin{equation*}
A_w(x):=A(x-w),\quad \text{for all } x\in\Hi.
\end{equation*}
\end{definition}

\begin{lemma}\label{fact:translation}
Let $A:\Hi\tto\Hi$ and let $w\in\Hi$. Then,
$$J_{(A_w)}=(J_{A})_w+w.$$
\end{lemma}
\begin{proof}
Observe that, for any $x\in\Hi$,
\begin{align*}
p\in J_{(A_w)}(x) & \Leftrightarrow x\in p+A(p-w)\\ & \Leftrightarrow x-w \in p-w+A(p-w) \Leftrightarrow p-w\in J_A(x-w),
\end{align*}
which proves the result.
\end{proof}

Next we collect some of the main convergence properties of a powerful algorithm for finding a zero of the sum of two maximally monotone operators, only involving individual evaluations of their resolvents. It is commonly called the \emph{Douglas--Rachford algorithm}, since it was originally proposed by J.~Douglas and H.H.~Rachford in~\cite{DR56} for solving a system of linear equations arising in heat conduction problems. However, Lions and Mercier~\cite{LM79} were the ones who successfully extended the algorithm to make it able to find a zero of the sum of two maximally monotone operators. We recommend~\cite[Appendix]{BLM17} to the reader interested in the connection between the original algorithm and the extension of Lions and Mercier.

\begin{fact}[Douglas--Rachford splitting algorithm]\label{fact:DR}
	Let $A,B:\Hi\tto \Hi$ be maximally monotone operators such that $\zer(A+B)\neq \emptyset$, let $\gamma>0$ and let ${(\lambda_n)}_{n=0}^\infty$ be a sequence in $[0,1]$ such that $\sum_{n\geq 0} \lambda_n(1-\lambda_n)=+\infty$. Given any $x_0\in\Hi$, set
	$$
	x_{n+1}=(1-\lambda_n)x_n+\lambda_nR_{\gamma B}R_{\gamma A}(x_n),\quad\text{for }n=0,1,2,\ldots.
	$$
	Then, there exists $x^\star\in\Fix\left( R_{\gamma B}R_{\gamma A}\right)$ such that  following assertions hold:
	\begin{enumerate}[label=(\roman*)]
		\item $({x_{n+1}-x_n)}_{n=0}^\infty$ converges strongly to $0$.
		\item ${(x_n)}_{n=0}^\infty$ converges weakly to $x^\star$, and $J_{\gamma A}(x^\star)\in\zer(A+B)$.
		\item $\left({J_{\gamma A}(x_n)}\right)_{n=0}^\infty$ converges weakly to $J_{\gamma A}(x^\star)$.
		\item Suppose that $A$ or $B$ is $\mu$-strongly monotone for some constant $\mu>0$. Then, the sequence $\left({J_{\gamma A}(x_n)}\right)_{n=0}^\infty$ converges strongly to the unique point in $\zer(A+B)$.
	\end{enumerate}
\end{fact}
\begin{proof}
See, e.g.,~\cite[Theorem~25.6]{BC11}.
\end{proof}

\section{The averaged alternating modified reflections method}\label{sec:AAMR}

We begin this section with the definition of a \emph{modified reflected resolvent}, which is the natural extension of the modified reflector introduced in~\cite[Definition~3.1]{AAMR}.

\begin{definition}
	Let $A:\Hi\tto\Hi$ be an operator. Given any $\beta\in{]0,1]}$, the operator $2\beta J_A-\Id$ is called a \emph{modified reflected resolvent} of $A$.
\end{definition}

The case $\beta=1$ coincides with the classical reflected resolvent $R_A$. In fact, for any $\beta\in{]0,1]}$, the modified reflected resolvent $2\beta J_A-\Id$ is a convex combination of $R_A$ and $-\Id$. Indeed, one has
$$(2\beta J_A-\Id)(x)=(1-\beta)(-x)+\beta R_A(x),\quad \text{for all } x\in\Hi.$$

In this work, our analysis is mainly based on the connection of the modified reflected resolvent with the classical reflected resolvent of a different operator, which is defined next.

\begin{definition}
Given an operator $A:\Hi\tto\Hi$ and given any $\beta\in{]0,1[}$, we define the \emph{$\beta$-strengthening} of $A$ as the operator $A^{(\beta)}:\Hi\tto \Hi$ defined by
\begin{equation*}
A^{(\beta)}(x):=\left( A+(1-\beta)\Id \right)\left(\frac{x}{\beta}\right),\quad \text{for all } x\in\Hi.
\end{equation*}
\end{definition}

\begin{proposition}\label{prop:strong_op}
Let $A:\Hi\tto \Hi$ be an operator and let $\beta\in{]0,1[}$. Then,
$$J_{A^{(\beta)}}=\beta J_{A}.$$
Further, $A$ is monotone if and only if ${A}^{(\beta)}$ is $\frac{1-\beta}{\beta}$-strongly monotone, and $A$ is maximally monotone if and only if $A^{(\beta)}$ is so.
\end{proposition}
\begin{proof}
The $\beta$-strengthening of $A$ can be expressed as
\begin{equation}\label{eq:Abeta_decom}
A^{(\beta)}=A\circ\left(\frac{1}{\beta}\Id\right)+\frac{1-\beta}{\beta}\Id,
\end{equation}
and thus,
\begin{equation}\label{eq:Abeta_decom2}
\Id+A^{(\beta)}=(\Id+A)\circ \left(\frac{1}{\beta} \Id\right).
\end{equation}
We directly deduce from~\eqref{eq:Abeta_decom2} that
\begin{equation}\label{eq:ran_equality}
\ran\left(\Id+A^{(\beta)}\right)=\ran\left( (\Id+A)\circ \left(\frac{1}{\beta} \Id\right) \right)=\ran\left(\Id+A\right).
\end{equation}
Now, for any $x\in\dom J_{A^{(\beta)}}=\dom J_{A}$ and $p\in\Hi$, we get from~\eqref{eq:Abeta_decom2} that
\begin{align*}
p\in J_{A^{(\beta)}}(x)& \Leftrightarrow x\in \left(\Id+A^{(\beta)}\right)(p) \Leftrightarrow  x\in \left(\Id+A\right)\left(\frac{p}{\beta}\right) \\
&\Leftrightarrow \frac{p}{\beta}\in J_A(x) \Leftrightarrow p\in\beta J_A(x),
\end{align*}
which proves that $J_{A^{(\beta)}}=\beta J_A$, as claimed.

By~\Cref{fact:prop_monotone}, $A$ is monotone if and only if $A \circ\left(\frac{1}{\beta}\Id\right)$ is monotone, so the assertion about the $\frac{1-\beta}{\beta}$-strong monotonicity of $A^{(\beta)}$ directly follows from~\eqref{eq:Abeta_decom}. Finally, $A$ is maximally monotone if and only if $A^{(\beta)}$ is so, according to~\Cref{fact:Minty} and~\eqref{eq:ran_equality}.
\end{proof}

\Cref{prop:strong_op} establishes strong monotonicity of $A^{(\beta)}$ when $A$ is monotone. The set of zeros of a strongly monotone operator is known to be at most a singleton (see, e.g.,~\cite[Corollary~23.35]{BC11}). Hence, the sum of the $\beta$-strengthenings of two monotone operators will have at most one zero. In the next proposition, we characterize this set for any pair of general operators.

\begin{proposition}\label{prop:zeros}
Let $A,B:\Hi\tto\Hi$ be two operators and let $\beta\in{]0,1[}$. Then, the set of zeros of the sum of their $\beta$-strengthenings $A^{(\beta)}$ and $B^{(\beta)}$ is given by
\begin{equation*}\label{eq:zeros}
\zer\left(A^{(\beta)}+B^{(\beta)}\right)=\beta J_{\frac{1}{2(1-\beta)}(A+B)} (0).
\end{equation*}
Consequently, $$\zer\left(A^{(\beta)}+B^{(\beta)}\right)\neq\emptyset\Leftrightarrow 0\in\ran\left(\Id+\frac{1}{2(1-\beta)}(A+B)\right).$$
\end{proposition}
\begin{proof}
For any $x\in\Hi$, one can easily check that
\begin{align*}
x\in\zer\left(A^{(\beta)}+B^{(\beta)}\right) & \Leftrightarrow 0\in A^{(\beta)}(x)+B^{(\beta)}(x)\\
& \Leftrightarrow 0\in A\left(\frac{x}{\beta}\right)+ B\left(\frac{x}{\beta}\right)+2(1-\beta)\frac{x}{\beta}\\
& \Leftrightarrow 0\in \frac{x}{\beta} + \frac{1}{2(1-\beta)}(A+B)\left(\frac{x}{\beta}\right)\\
& \Leftrightarrow \frac{x}{\beta}\in J_{\frac{1}{2(1-\beta)}(A+B)} (0) \Leftrightarrow x\in \beta J_{\frac{1}{2(1-\beta)}(A+B)} (0),
\end{align*}
which proves the result.
\end{proof}

We are ready to prove our main result, which shows that the AAMR method can be applied to compute the resolvent of the sum of two maximally monotone operators.

\begin{theorem}[AAMR splitting algorithm]\label{th:AAMR}
Let $A,B:\Hi\tto \Hi$ be two maximally monotone operators, let $\gamma>0$ and let ${(\lambda_n)}_{n=0}^\infty$ be a sequence in $[0,1]$ such that $\sum_{n\geq 0} \lambda_n(1-\lambda_n)=+\infty$. Let $\beta\in{]0,1[}$ and suppose that ${q\in\ran\left(\Id+\frac{\gamma}{2(1-\beta)}(A+B)\right)}$. Given any $x_0\in\Hi$, for every $n=0,1,2,\ldots$, set
\begin{equation}\label{eq:AAMR}
x_{n+1}=(1-\lambda_n)x_n+\lambda_n(2\beta J_{(\gamma B_{-q})}-\Id)(2\beta J_{(\gamma A_{-q})}-\Id)(x_n).
\end{equation}
Then, there exists $x^\star\in\Fix\left( (2\beta J_{(\gamma B_{-q})}-\Id)(2\beta J_{(\gamma A_{-q})}-\Id)\right)$ such that the following hold:
	\begin{enumerate}[label=(\roman*)]
		\item $({x_{n+1}-x_n)}_{n=0}^\infty$ converges strongly to $0$;\label{th:AAMR_i}
		\item ${(x_n)}_{n=0}^\infty$ converges weakly to $x^\star$, and $J_{\gamma A}(q+x^\star)=J_{\frac{\gamma}{2(1-\beta)}(A+B)}(q)$;\label{th:AAMR_ii}
		\item $\left({J_{\gamma A}(q+x_n)}\right)_{n=0}^\infty$ converges strongly to $J_{\frac{\gamma}{2(1-\beta)}(A+B)}(q)$.\label{th:AAMR_iii}
	\end{enumerate}
\end{theorem}
\begin{proof}
Since $A$ and $B$ are maximally monotone,  by~\Cref{fact:prop_monotone}, the operators $\gamma A_{-q}$ and $\gamma B_{-q}$ are also maximally monotone. Thus, in view of~\Cref{prop:strong_op}, the iterative scheme in~\eqref{eq:AAMR} becomes
$$
x_{n+1}=(1-\lambda_n)x_n+\lambda_nR_{(\gamma B_{-q})^{(\beta)}}R_{(\gamma A_{-q})^{(\beta)}}(x_n),\quad n=0,1,2,\ldots,
$$
with $(\gamma A_{-q})^{(\beta)}$ and $(\gamma B_{-q})^{(\beta)}$ maximally monotone and $\frac{\beta}{1-\beta}$-strongly monotone. Now observe that $q\in\ran\left(\Id+\frac{\gamma}{2(1-\beta)}(A+B)\right)$ if and only if there exists $z\in\Hi$ such that
\begin{align*}
q\in z+\frac{\gamma}{2(1-\beta)}\left( A+B\right)(z) & \Leftrightarrow 0\in z-q+\frac{1}{2(1-\beta)}\gamma\left( A_{-q}+B_{-q}\right)(z-q)\\
& \Leftrightarrow 0\in\ran\left(\Id+\frac{1}{2(1-\beta)}(\gamma A_{-q}+\gamma B_{-q})\right).
\end{align*}
Hence,~\Cref{prop:zeros} implies
\begin{equation}\label{eq:AAMR_zeros}
\zer\left((\gamma A_{-q})^{(\beta)}+(\gamma B_{-q})^{(\beta)}\right)=\left\{\beta J_{\frac{\gamma}{2(1-\beta)}(A_{-q}+B_{-q})} (0) \right\}\neq\emptyset.
\end{equation}
We are then in position to apply~\cref{fact:DR}, which yields the existence of
$$x^\star\in \Fix\left( R_{(\gamma B_{-q})^{(\beta)}}R_{(\gamma A_{-q})^{(\beta)}} \right)=\Fix\left( (2\beta J_{(\gamma B_{-q})}-\Id)(2\beta J_{(\gamma A_{-q})}-\Id)\right)$$
such that $({x_{n+1}-x_n)}_{n=0}^\infty\to 0$, ${(x_n)}_{n=0}^\infty\rightharpoonup x^\star$ and
\begin{equation}\label{eq:AAMR_fixpoint}
J_{(\gamma A_{-q})^{(\beta)}}(x^\star)\in \zer\left((\gamma A_{-q})^{(\beta)}+(\gamma B_{-q})^{(\beta)}\right).
\end{equation}
According to~\Cref{prop:strong_op}, together with~\Cref{fact:translation}, we have that
\begin{equation}\label{eq:AAMR_untranslate}
J_{(\gamma A_{-q})^{(\beta)}}(x)=\beta J_{(\gamma A_{-q})}(x)=\beta  \left(J_{\gamma A}(x+q)-q\right), \quad \text{for all } x\in\Hi;
\end{equation}
and also by~\Cref{fact:translation},
\begin{equation}\label{eq:AAMR_untranslate_solution}
J_{\frac{\gamma}{2(1-\beta)}(A_{-q}+B_{-q})} (0)=J_{\left(\left(\frac{\gamma}{2(1-\beta)}(A+B)\right)_{-q}\right)} (0)=J_{\frac{\gamma}{2(1-\beta)}(A+B)} (q)-q.
\end{equation}
Therefore, by combining~\eqref{eq:AAMR_zeros}, \eqref{eq:AAMR_fixpoint}, \eqref{eq:AAMR_untranslate} and \eqref{eq:AAMR_untranslate_solution}, we get that $$J_{\gamma A}(q+x^\star)=J_{\frac{\gamma}{2(1-\beta)}(A+B)}(q),$$and thus statements~\cref{th:AAMR_i,th:AAMR_ii} have been proved. Finally, thanks to the strong monotonicity of $(\gamma A_{-q})^{(\beta)}$ or $(\gamma B_{-q})^{(\beta)}$, \Cref{fact:DR} asserts that the sequence $\big({J_{(\gamma A_{-q})^{(\beta)}}(x_n)}\big)_{n=0}^\infty$ converges strongly to the unique zero of $(\gamma A_{-q})^{(\beta)}+(\gamma B_{-q})^{(\beta)}$. Again, taking into account \eqref{eq:AAMR_zeros}, \eqref{eq:AAMR_untranslate} and \eqref{eq:AAMR_untranslate_solution}, this is equivalent to
\begin{equation*}
\left(\beta\left({J_{\gamma A}(q+x_n)}-q\right)\right)_{n=0}^\infty \quad\rightarrow\quad \beta\left(J_{\frac{\gamma}{2(1-\beta)}(A+B)}(q)-q\right),
\end{equation*}
which implies \cref{th:AAMR_iii} and completes the proof.
\end{proof}

As a direct consequence of~\Cref{th:AAMR} we derive the next result, which corresponds to \cite[Theorem~4.1]{AAMR}, and establishes the convergence of the AAMR method when we turn from resolvents to projectors.

\begin{corollary}[AAMR for best approximation problems]\label{cor:AAMR_projections}
	Consider two nonempty, closed and convex sets $A,B\subseteq\Hi$. Let ${(\lambda_n)}_{n=0}^\infty$ be a sequence in $[0,1]$ such that $\sum_{n\geq 0} \lambda_n(1-\lambda_n)=+\infty$ and fix any $\beta\in{]0,1[}$. Given $q\in\Hi$, choose any  $x_0\in\Hi$ and consider the sequence defined by
	\begin{equation*}
	x_{n+1}=(1-\lambda_n)x_n+\lambda_n(2\beta P_{B-q}-\Id)(2\beta P_{A-q}-\Id)(x_n), \quad n=0,1,2,\ldots.
	\end{equation*}
	Then, if  $A\cap B\neq\emptyset$ and $q-P_{A\cap B}(q)\in (N_A+N_B)\left(P_{A\cap B}(q)\right)$, the following assertions hold:
	\begin{enumerate}[label=(\roman*)]
	\item ${(x_{n+1}-x_n)}_{n=0}^\infty$ is strongly convergent to $0$;
	\item ${(x_n)}_{n=0}^\infty$ is weakly convergent to a point $$x^\star\in \Fix \left((2\beta P_{B-q}-\Id)(2\beta P_{A-q}-\Id)\right)$$ such that
	$
	P_A(q+x^\star)=P_{A\cap B}(q);
	$
	\item ${\left(P_A(q+x_n)\right)}_{n=0}^\infty$ is strongly convergent to $P_{A\cap B}(q)$.
	\end{enumerate}	
\end{corollary}
\begin{proof}
We know from Examples~\ref{ex:subdif_normalcone}\cref{ex:normalcone} and \ref{ex:prox_proj}\cref{ex:proj} that the normal cones $N_A$ and $N_B$ are maximally monotone operators with $J_{N_A}=P_A$ and $J_{N_B}=P_B$. Moreover, it can be easily checked that the normal cones to the displaced sets $A-q$ and $B-q$ coincide with the inner $(-q)$-perturbations of $N_A$ and $N_B$, i.e.,
\begin{equation*}
N_{(A-q)}=\left(N_A\right)_{-q} \quad \text{and} \quad N_{(B-q)}=\left(N_B\right)_{-q}.
\end{equation*}
Therefore, according to~\Cref{ex:prox_proj}\cref{ex:proj}, it holds that $J_{\left(\left(N_A\right)_{-q}\right)}=P_{A-q}$ and $J_{\left(\left(N_B\right)_{-q}\right)}=P_{B-q}$. Now observe that
\begin{equation*}
q-P_{A\cap B}(q)\in \left(N_A+N_B\right)\left(P_{A\cap B}(q)\right)=\frac{1}{2(1-\beta)}(N_A+N_B)\left(P_{A\cap B}(q)\right),
\end{equation*}
which implies that $q\in\ran\left(\Id+\frac{1}{2(1-\beta)}(N_A+N_B)\right)$ and $$P_{A\cap B}(q)=J_{\frac{1}{2(1-\beta)}(N_A+N_B)}(q).$$ Hence, the result follows from applying~\Cref{th:AAMR} to $N_A$ and $N_B$, with $\gamma=1$.
\end{proof}

\begin{example}[Proximity operator of the sum of two functions] Given two proper lower semicontinuous convex functions $f,g:\Hi\to{]-\infty,+\infty]}$, \cref{th:AAMR} can be applied to their subdifferentials $\partial f$ and $\partial g$. Hence, given a point $q\in\Hi$, this gives rise to a sequence $(x_n)_{n=0}^\infty$ such that
$$\left(\Prox_f(q+x_n)\right)_{n=0}^\infty \quad \to \quad \Prox_{\frac{1}{2(1-\beta)}(f+g)}(q),$$
provided that
\begin{equation}\label{eq:sumproxcond}
q\in\ran\left(\Id+\frac{1}{2(1-\beta)}(\partial f + \partial g)\right).
\end{equation}
Note that the latter holds for all $q\in\Hi$ when $\partial f+\partial g=\partial (f+g)$, so a sufficient condition for~\eqref{eq:sumproxcond} is
$$0\in\sri(\dom f-\dom g),$$
(see, e.g.,~\cite[Corollary~16.38]{BC11}), where $\sri$ stands for the \emph{strong relative interior}.
\end{example}

\section{Parallel AAMR splitting for the resolvent of a finite sum}\label{sec:Parallel}

In this section, we discuss how to implement the AAMR scheme to compute the resolvent of a finite sum of maximally monotone operators. Given a collection of $r$ operators $A_i:\Hi\tto\Hi$, $i=1,2,\ldots,r$, and $q\in\ran\left(\Id+\sum_{i=1}^rA_i\right)$, the problem of interest is now
\begin{equation}\label{eq:par_problem}
\text{Find } p\in J_{\sum_{i=1}^rA_i}(q).
\end{equation}
To transform this problem into a two-operators sum problem, we turn to the following standard product space reformulation, which was originally proposed by Pierra~\cite{Pierra}. Consider the product Hilbert space $\bHi:=\Hi^r=\Hi\times\stackrel{(r)}{\cdots}\times\Hi$, and define the operator $\bs{B}:\bHi\tto\bHi$ by
\begin{equation}\label{eq:prod_operatorB}
\bs{B}(\bs{x}):=A_1(x_1)\times A_2(x_2)\times \cdots \times A_r(x_r), \quad \forall \bs{x}=(x_1,x_2,\ldots,x_r)\in\bHi,
\end{equation}
and the set $\bs{D}:=\{(x,\ldots,x)\in\bHi : x\in\Hi\}$, commonly known as the \emph{diagonal}. We denote by $\bs{j}:\Hi\to\bs{D}$ the canonical embedding that maps any $x\in\Hi$ to $\bs{j}(x)=(x,x,\ldots,x)\in\bs{D}$.

The following result collects the fundamentals of the product space reformulation.

\begin{fact}\label{fact:prod_space}
The following hold:
\begin{enumerate}[label=(\roman*)]
\item The resolvent of $\bs{B}$ can be computed as
\begin{equation*}
J_{\bs{B}}(\bs{x})=J_{A_1}(x_1)\times J_{A_2}(x_2)\times\cdots\times J_{ A_r}(x_r), \quad \forall \bs{x}=(x_1,x_2,\ldots,x_r)\in\bHi.
\end{equation*}
Further, the operator $\bs{B}$ is (maximally) monotone whenever $A_1, A_2,\ldots, A_r$ are so.

\item The normal cone to $\bs{D}$ is given by
\begin{equation*}
N_{\bs{D}}(\bs{x}) =\left\{\begin{array}{ll}\{\bs{u}=(u_1,u_2,\ldots,u_r)\in\bHi : \sum_{i=1}^r u_i=0 \}, &\text{if }\bs{x}\in \bs{D},\\
\emptyset, & \text{otherwise.}\end{array}\right.
\end{equation*}
It is a maximally monotone operator and
\begin{equation*}
J_{N_{\bs{D}}}(\bs{x})=P_{\bs{D}}(\bs{x})=\bs{j}\left(\frac{1}{r}\sum_{i=1}^r x_i\right), \quad \forall \bs{x}=(x_1,x_2,\ldots,x_r)\in\bHi.
\end{equation*}
\item $\zer\left( \bs{B}+N_{\bs{D}}\right)=\bs{j}\left(\zer\left(\sum_{i=1}^r A_i\right)\right)$.\label{fact:prod_spaceIII}
\end{enumerate}
\end{fact}
\begin{proof}
See, e.g.,~\cite[Proposition~25.4]{BC11}.
\end{proof}

According to the previous result, the product space reformulation is a powerful trick for reducing the problem of finding zeros of the sum of finitely many operators to an equivalent problem involving only two, while keeping their monotonicity properties.  As we show next, it turns out to be very useful in our context, where we are interested in computing the resolvent of the sum.

\begin{proposition}\label{prop:par_resolv}
For any $\bs{x}=(x,x,\ldots,x)\in\bs{D}$, we have
\begin{equation*}
J_{\bs{B}+N_{\bs{D}}}(\bs{x})=\bs{j}\left(J_{\frac{1}{r}\sum_{i=1}^{r}A_i}\left( x \right)\right).
\end{equation*}
Consequently,
\begin{equation*}
\ran\left(\Id+\bs{B}+N_{\bs{D}}\right)\cap\bs{D}=\bs{j}\left(\ran\left(\Id+\frac{1}{r}\sum_{i=1}^rA_i\right)\right).
\end{equation*}
\end{proposition}
\begin{proof}
Fix some $\bs{x}=(x,x,\ldots,x)\in\bs{D}$. To prove the direct inclusion, pick any $\bs{p}\in J_{\bs{B}+N_{\bs{D}}}(\bs{x})$. Then, we have that
\begin{equation*}
\bs{x}\in \bs{p}+\bs{B}(\bs{p})+N_{\bs{D}}(\bs{p}).
\end{equation*}
This ensures the nonemptyness of $N_{\bs{D}}(\bs{p})$, and then it necessarily holds that $\bs{p}=\bs{j}(p)\in\bs{D}$, for some $p\in\Hi$. Moreover, there must exist some vector $\bs{u}=(u_1,u_2,\ldots,u_r)\in\bHi$ with $\sum_{i=1}^ru_i=0$ such that
\begin{equation*}
x \in p + A_i(p)+u_i,\quad\text{for all } i=1,2,\ldots,r.
\end{equation*}
Thus, by adding up all these equations and dividing by $r$, we deduce that
$x\in p+\frac{1}{r}\sum_{i=1}^rA_i(p)$,
or equivalently, that $p\in J_{\frac{1}{r}\sum_{i=1}^{r}A_i}\left( x \right)$.  

To prove the reverse inclusion, take any $\bs{p}=\bs{j}(p)$ with $p\in J_{\frac{1}{r}\sum_{i=1}^{r}A_i}\left( x \right)$. Then, for each $i=1,2,\ldots,r$, there exists $a_i\in A_i(p)$ such that
\begin{equation*}
x= p+\frac{1}{r}\sum_{i=1}^ra_i \Leftrightarrow r(x-p)-\sum_{i=1}^ra_i=0 \Leftrightarrow \sum_{i=1}^r(x-p-a_i)=0.
\end{equation*}
Let $\bs{a}:=(a_1,a_2,\ldots,a_r)$ and $\bs{u}:=(u_1,u_2,\ldots,u_r)$, where $u_i:=x-p-a_i$, for each $i=1,2,\ldots,r$. By construction, we get that $\bs{x}=\bs{p}+\bs{a}+\bs{u}$, with $\bs{a}\in\bs{B}{(\bs{p})}$ and $\bs{u}\in N_{\bs{D}}(\bs{p})$. This implies $\bs{p}\in J_{\bs{B}+N_{\bs{D}}}(\bs{x})$, which completes the proof.
\end{proof}

Thanks to~\Cref{prop:par_resolv}, problem~\eqref{eq:par_problem} can be fitted within the framework of~\Cref{th:AAMR}, allowing us to derive the following parallel splitting algorithm.

\begin{theorem}[Parallel AAMR splitting algorithm]\label{th:par_AAMR}
Let $A_i:\Hi\tto \Hi$ be maximally monotone operators for $i=1,2,\ldots,r$, let $\gamma>0$ and let ${(\lambda_n)}_{n=0}^\infty$ be a sequence in $[0,1]$ such that $\sum_{n\geq 0} \lambda_n(1-\lambda_n)=+\infty$. Let $\beta\in{]0,1[}$ and suppose that $q\in\ran\left(\Id+\frac{\gamma}{2r(1-\beta)}\sum_{i=1}^rA_i\right)$. Given $x_{1,0},x_{2,0},\ldots,x_{r,0}\in\Hi$, set
\begin{equation}\label{eq:par_AAMR_alg}
\begin{aligned}
&\text{for } n=0, 1, 2, \ldots:\\
&\left\lfloor\begin{array}{l}
p_n=\frac{1}{r}\sum_{i=1}^r x_{i,n},\\
\text{for } i=1, 2,\ldots, r:\\
\left\lfloor\begin{array}{l}
x_{i,n+1}=(1-\lambda_n)x_{i,n}+\lambda_n\left(2\beta J_{\left(\gamma (A_i)_{-q}\right)}-\Id\right)\left(2\beta p_n -x_{i,n}\right).
\end{array}\right.
\end{array}\right.
\end{aligned}
\end{equation}
Then, the following hold:
	\begin{enumerate}[label=(\roman*)]
		\item $({{x}_{i,n+1}-{x}_{i,n})}_{n=0}^\infty$ converges strongly to $0$, for all $i=1,2,\ldots,r$;
		\item ${({x}_{i,n})}_{n=0}^\infty$ converges weakly to ${x_i}^\star\in\Hi$, for all $i=1,2,\ldots,r$, and
$$q+\frac{1}{r}\sum_{i=1}^rx_i^\star=J_{\frac{\gamma}{2r(1-\beta)}\sum_{i=1}^rA_i}(q);$$
		\item $\left(q+p_n\right)_{n=0}^\infty$ converges strongly to $J_{\frac{\gamma}{2r(1-\beta)}\sum_{i=1}^rA_i}(q)$.
	\end{enumerate}
\end{theorem}
\begin{proof}
Let $\bs{B}$ be the operator defined as in~\eqref{eq:prod_operatorB}, and consider the normal cone to the diagonal set $N_{\bs{D}}$. By~\Cref{fact:prod_space}, both operators are maximally monotone. For each $n=0,1,\ldots$, set
$
\bs{x}_n:=(x_{1,n},x_{2,n},\ldots,x_{r,n})\in\bHi$ and $\bs{p}_n:=\bs{j}(p_n)\in\bs{D}.
$
Observe that $\bs{p}_n=P_{\bs{D}}(\bs{x}_n)=J_{N_{\bs{D}}}(\bs{x}_n)$. Further, set $\bs{q}:=\bs{j}(q)$ and note that, since $\bs{D}$ is a linear subspace and $\bs{q}\in\bs{D}$, we have
$$N_{\bs{D}}=N_{\bs{D}-\bs{q}}=\left(N_{\bs{D}}\right)_{-\bs{q}}=\left(\gamma N_{\bs{D}}\right)_{-\bs{q}}.$$
Therefore, the iterative scheme in~\eqref{eq:par_AAMR_alg} can be expressed as
\begin{equation*}
\bs{x}_{n+1}=(1-\lambda_n)\bs{x}_n+\lambda_n\left(2\beta J_{\left(\gamma \bs{B}_{-\bs{q}}\right)}-\Id\right)\left(2\beta J_{\left(\gamma(N_{\bs{D}})_{-\bs{q}}\right)}-\Id\right)(\bs{x}_n),
\end{equation*}
for $n=0,1,2,\ldots$.
According to~\Cref{prop:par_resolv}, we have that
\begin{equation*}
q\in\ran\left(\Id+\frac{\gamma}{2r(1-\beta)}\sum_{i=1}^rA_i\right) \Leftrightarrow \bs{q}\in \ran\left(\Id+\frac{\gamma}{2(1-\beta)}(\bs{B}+N_{\bs{D}})\right),
\end{equation*}
and
\begin{equation*}
J_{\frac{\gamma}{2(1-\beta)}(\bs{B}+N_{\bs{D}})}(\bs{q})=\bs{j}\left(J_{\frac{\gamma}{2r(1-\beta)}\sum_{i=1}^{r}A_i}\left( q\right)\right).
\end{equation*}
Finally, note that for any $\bs{x}=(x_1,x_2,\ldots,x_r)\in\bHi$, the shadows can be expressed as
\begin{equation*}
J_{\gamma N_{\bs{D}}}(\bs{q}+\bs{x})=P_{\bs{D}}(\bs{q}+\bs{x})=\bs{j}\left(q+\frac{1}{r}\sum_{i=1}^r x_i\right).
\end{equation*}
In particular, $J_{\gamma N_{\bs{D}}}(\bs{q}+\bs{x}_n)=\bs{j}(q+p_n)$. Hence, the result follows from applying~\Cref{th:AAMR} to $\bs{B}$ and $N_{\bs{D}}$.
\end{proof}

\begin{remark}
As in~\Cref{cor:AAMR_projections}, if we choose the operators involved in~\Cref{th:par_AAMR} to be the normal cones to $r$ closed and convex sets $C_1,C_2,\ldots,C_r\subseteq\Hi$ satisfying
\begin{equation*}
q-P_{\bigcap_{i=1}^r C_i}(q)\in \sum_{i=1}^r N_{C_i}\left(P_{\bigcap_{i=1}^r C_i}(q)\right),
\end{equation*}
we can deduce the AAMR algorithm  established in~\cite[Theorem~5.1]{AAMR} for finding the projection of the point  $q$ onto the intersection of finitely many sets.
\end{remark}
\subsection{An alternative parallel splitting}

A different parallel algorithm for solving~\eqref{eq:par_problem}, involving modified reflected resolvents of the operators, can be constructed. Recall that the AAMR method for two operators (\Cref{th:AAMR}) has been shown to be, in essence, a Douglas--Rachford iteration for finding a zero of the sum of the $\beta$-strengthenings. The following result is a generalization of~\Cref{prop:zeros}, and characterizes the set of zeros of the sum of the $\beta$-strengthenings of a finite collection of operators. The proof is completely analogous so it is omitted.

\begin{proposition}\label{prop:par_zeros}
Let $A_i:\Hi\tto \Hi$ be some operators  for $i=1,2,\ldots,r$ and let $\beta\in{]0,1[}$. Then, the set of zeros of the sum of their $\beta$-strengthenings is given by
\begin{equation*}\label{eq:par_zeros}
\zer\left(\sum_{i=1}^rA_i^{(\beta)}\right)=\beta J_{\frac{1}{r(1-\beta)}\sum_{i=1}^rA_i} (0).
\end{equation*}
Therefore, $\zer\left(\sum_{i=1}^rA_i^{(\beta)}\right)\neq\emptyset$ if and only if $ 0\in\ran\left(\Id+\frac{1}{r(1-\beta)}\sum_{i=1}^rA_i\right).$
\end{proposition}

In view of the previous proposition, we derive the following alternative splitting
algorithm for computing the resolvent of a finite sum of maximally monotone operators.

\begin{theorem}[Alternative parallel AAMR-like splitting algorithm]\label{th:par_AAMR2}%
Let $A_i:\Hi\tto \Hi$ be maximally monotone operators for $i=1,2,\ldots,r$,  let ${\gamma>0}$ and let ${(\lambda_n)}_{n=0}^\infty$ be a sequence in $[0,1]$ such that $\sum_{n\geq 0} \lambda_n(1-\lambda_n)=+\infty$. Let $\beta\in{]0,1[}$ and suppose that $q\in\ran\left(\Id+\frac{\gamma}{r(1-\beta)}\sum_{i=1}^rA_i\right)$. Given $x_{1,0},x_{2,0},\ldots,x_{r,0}\in\Hi$, set
\begin{equation}\label{eq:par_AAMR2_alg}
\begin{aligned}
&\text{for } n=0, 1, 2, \ldots:\\
&\left\lfloor\begin{array}{l}
p_n=\frac{1}{r}\sum_{i=1}^r x_{i,n},\\
\text{for } i=1, 2,\ldots, r:\\
\left\lfloor\begin{array}{l}
x_{i,n+1}=(1-\lambda_n)x_{i,n}+\lambda_n\left(2\beta J_{\left(\gamma (A_i)_{-q}\right)}-\Id\right)\left(2 p_n -x_{i,n}\right).
\end{array}\right.
\end{array}\right.
\end{aligned}
\end{equation}
Then, the following hold:
	\begin{enumerate}[label=(\roman*)]
		\item $({{x}_{i,n+1}-{x}_{i,n})}_{n=0}^\infty$ converges strongly to $0$, for all $i=1,2,\ldots,r$;
		\item ${({x}_{i,n})}_{n=0}^\infty$ converges weakly to ${x_i}^\star\in\Hi$, for all $i=1,2,\ldots,r$, and
$$q+\frac{1}{\beta r}\sum_{i=1}^rx_i^\star=J_{\frac{\gamma}{r(1-\beta)}\sum_{i=1}^rA_i}(q);$$
		\item $\left(q+\frac{1}{\beta} p_n\right)_{n=0}^\infty$ converges strongly to $J_{\frac{\gamma}{r(1-\beta)}\sum_{i=1}^rA_i}(q)$.
	\end{enumerate}
\end{theorem}
\begin{proof} After rewriting the iterative scheme in~\eqref{eq:par_AAMR2_alg} as
$$
\bs{x}_{n+1}=(1-\lambda_n)\bs{x}_n+\lambda_nR_{(\gamma \bs{B}_{-\bs{q}})^{(\beta)}}R_{N_{\bs{D}}}(\bs{x}_n),\quad\text{for }n=0,1,2,\ldots,
$$
the proof is analogous to that of~\Cref{th:AAMR}, using \Cref{prop:par_zeros} together with \Cref{fact:prod_space}\cref{fact:prod_spaceIII} instead of \Cref{prop:zeros}.
\end{proof}

\subsection{Numerical experiment}
We conclude with a simple numerical experiment, where we compare the two parallel variants of the AAMR splitting algorithm presented in \Cref{th:par_AAMR,th:par_AAMR2}.
We consider the quadratic best approximation problem of finding the closest point to the origin in the intersection of $N$ balls in $\R^{10}$:
\begin{equation*}
\text{Find } P_{\bigcap_{i=1}^N {B}_i}(0),\quad \text{with } {B}_i=:\{x\in\R^{10} : \|x-c_i\|\leq r_i\},\, i=1,2,\ldots,N.
\end{equation*}
The iteration generated by each of the two algorithms for three balls in $\R^2$ is illustrated in~\Cref{fig:difference}.
\begin{figure}\centering
\subfigure[Splitting algorithm in \Cref{th:par_AAMR}: the sequence $(p_n)_{n=0}^\infty$ converges to $P_{B_1\cap B_2\cap B_3}(0)$]{\hspace{30pt}\scalebox{.44}{\input{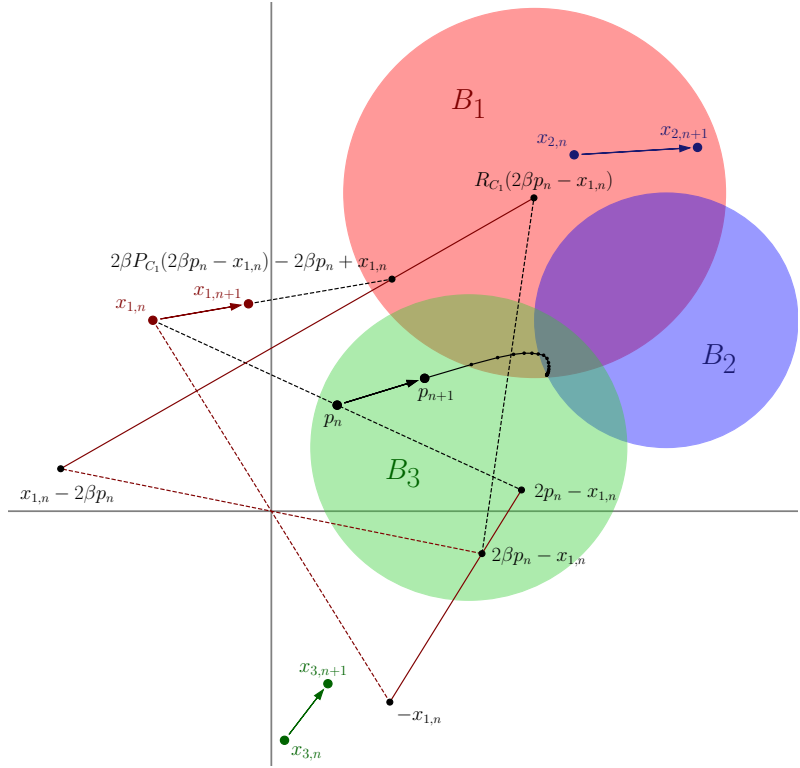}}\hspace{30pt}}
\subfigure[Splitting algorithm in \Cref{th:par_AAMR2}: the sequence $\left(\frac{1}{\beta}p_n\right)_{n=0}^\infty$ converges to $P_{B_1\cap B_2\cap B_3}(0)$]{\hspace{35pt}\scalebox{.44}{\input{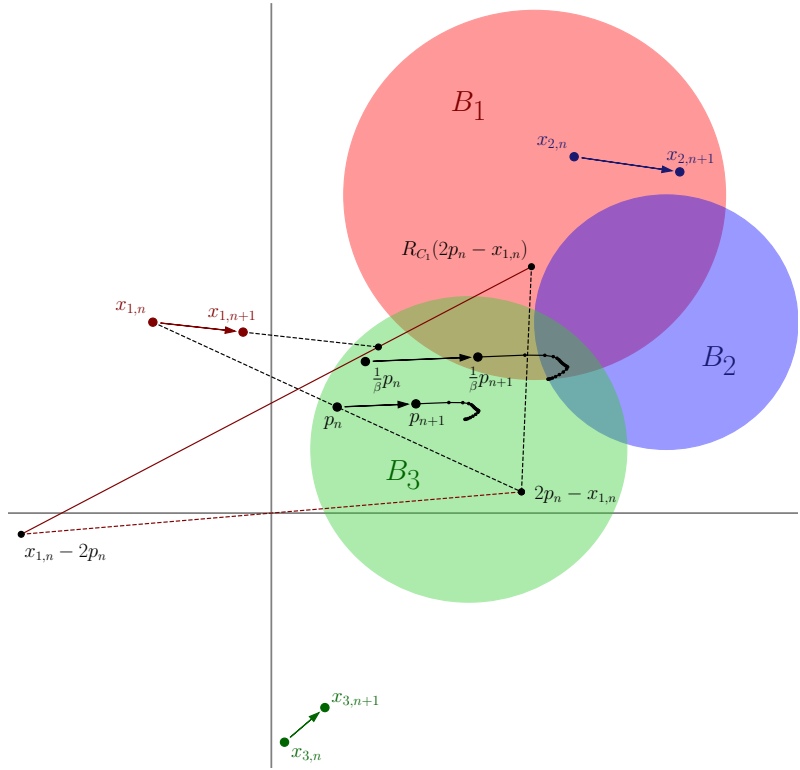}}\hspace{35pt}}
\caption{Illustration of the computation of the iterations of the splitting algorithms proposed, when they are applied to the normal cones of three balls $B_1,B_2,B_3\subset\mathbb{R}^2$, with $q=0$, $\lambda_n=0.4$ and $\beta=0.7$}\label{fig:difference}
\end{figure}

In the experiment, for each value of $N\in\{2,4,6,8,10\}$, $100$ feasible problems were randomly generated as follows. First, we picked a point $z$ with coordinates randomly uniformly generated in the range $[-5, 5]$. Then, for each $i=1,2,\ldots,N$, a point $b_i$ was randomly chosen with coordinates in $[-5,5]$, and the center $c_i$ of each ball $B_i$ was set to $c_i:=z+b_i$. Finally, a radius
$r_i:=\|b_i\|+\alpha_i$ was defined by adding to the center's distance from $z$, a random number $\alpha_i$ uniformly picked from the range $[0.05,0.1]$. In this way, the point $z$ is in the interior of every ball, thus, yielding a consistent best approximation problem for which the convergence of the algorithms is guaranteed.

In our test, we fixed $\lambda_n=0.9$, which seems to be a sensible choice for both algorithms. For each problem and each value of $\beta\in\{0.5,0.505,\ldots,0.99,0.995\}$ (values of $\beta<0.5$ were dominated by $0.5$), both algorithms were run from a random starting point with coordinates in $[-5,5]$. We used a stopping criterion based on the true error; i.e.,  the algorithm in~\Cref{th:par_AAMR} (which we refer to as \emph{original}) was stopped when
\begin{equation*}
\left\|p_n-P_{\bigcap_{i=1}^N {B}_i}(0)\right\|<10^{-6},
\end{equation*}
and the algorithm in~\Cref{th:par_AAMR2} (which we refer to as \emph{alternative}), when
\begin{equation*}
\left\|\frac{1}{\beta}p_n-P_{\bigcap_{i=1}^N {B}_i}(0)\right\|<10^{-6}.
\end{equation*}

The results of the experiment are displayed in~\Cref{fig:experiment}, where we can observe that the behavior of both algorithms is very similar. For each algorithm, there exists an optimal value of $\beta$, which depends on the number of constraints, minimizing the number of iterations needed to converge. The alternative algorithm outperforms the original one for small values of $\beta$, while the opposite occurs for large values. The value of $\beta$ on which the switch takes place increases with the number of constraints. In practice, as this value would be in principle unknown, the alternative parallel algorithm applied with some $\beta\in[0.85,0.95]$ is preferable in this setting.

\begin{figure}[ht!]
\centering
\includegraphics[width=0.8\textwidth]{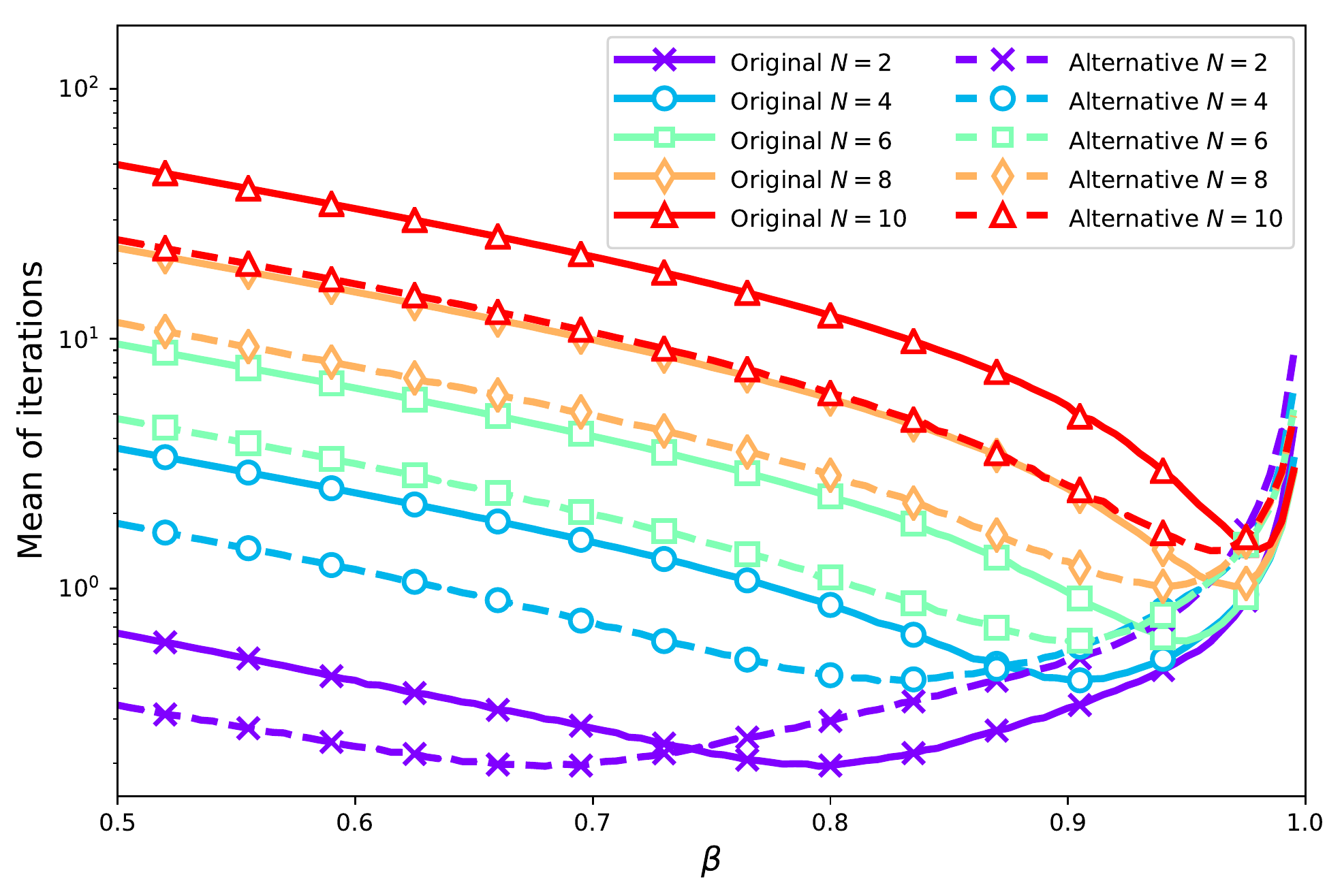}\\[12pt]
\includegraphics[width=0.8\textwidth]{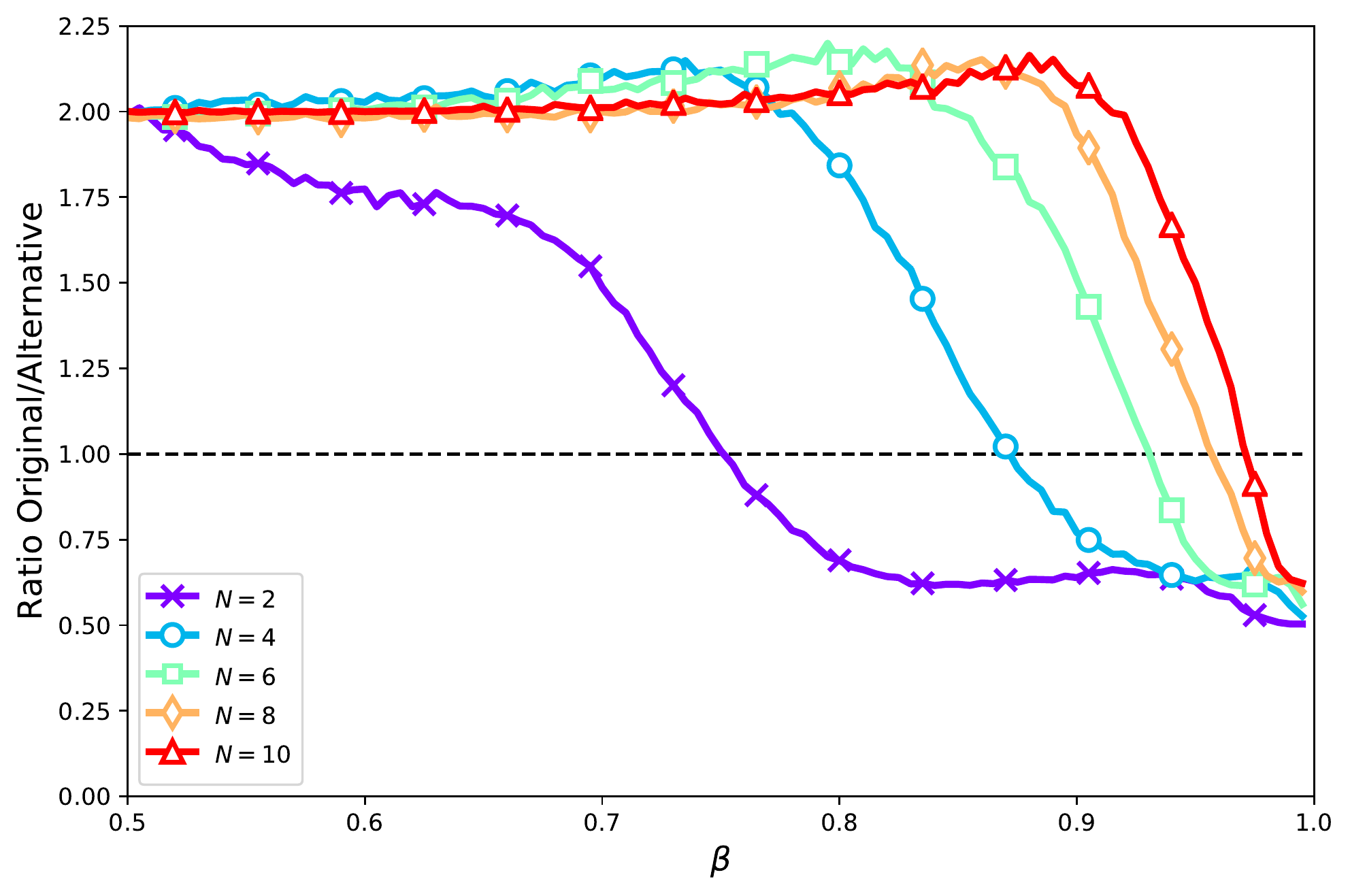}\\[6pt]
\caption{Results of the numerical experiment comparing the algorithms in~\Cref{th:par_AAMR,th:par_AAMR2}. In the top figure, we show the average number of iterations required by each algorithm with respect to the value of $\beta$. In the bottom figure, we show the ratio between the average number of iterations required by the original algorithm  and the alterative one, for each number of constraints with respect to the value of $\beta$.}\label{fig:experiment}
\end{figure}

\paragraph{Acknowledgements} We greatly appreciate the constructive comments of two anonymous reviewers which helped us to improve the paper. This work was partially supported by  Ministerio de Econom\'ia, Industria y Competitividad (MINECO) of Spain  and  European Regional Development Fund (ERDF), grant MTM2014-59179-C2-1-P. FJAA was supported by the Ram\'on y Cajal program by MINECO  and  ERDF (RYC-2013-13327) and RC was supported by MINECO and European Social Fund (BES-2015-073360) under the program ``Ayudas para contratos predoctorales para la formaci\'on de doctores 2015''.

\end{document}